\documentclass[11pt]{amsart}
\usepackage{amsmath}
\usepackage{amsfonts}
\usepackage{amssymb}
\newtheorem{thm}{Theorem}[section]

\newtheorem{cor}[thm]{Corollary}

\newtheorem{remark}[thm]{Remark}
\newtheorem{lemma}[thm]{Lemma}
\newtheorem{prop}[thm]{Proposition}

\newtheorem{exam}[thm]{Example}
\newtheorem{defn}[thm]{Definition}

\newcommand{\bb}[1]{\mathbb{#1}}
\newcommand{\cl}[1]{\mathcal{#1}}

\newcommand{\inner}[2]{\left\langle {#1}, {#2} \right\rangle}

\begin{document}

\title[Quantum Graph Homomorphisms via Operator Systems]{Quantum Graph Homomorphisms via Operator Systems}

\author[C.~M.~Ortiz]{Carlos M. Ortiz}
\address{Pacific Northwest National Laboratory,
Richland, WA 99352, U.S.A.}
\email{carlos.ortiz@pnnl.gov}
\author[V.~I.~Paulsen]{Vern I.~Paulsen}
\address{IQC and Department of Pure Mathematics, University of Waterloo,
Waterloo, Ontario, N2L 3G1, Canada}
\email{vpaulsen@uwaterloo.ca}

\date{\today}

\begin{abstract} We explore the concept of a graph homomorphism through the lens of C$^*$-algebras and operator systems. We start by studying the various notions of a quantum graph homomorphism and examine how they are related to each other. We then define and study a C$^*$-algebra that encodes all the information about these homomorphisms and establish a connection between computational complexity and the representation of these algebras. We use this C$^*$-algebra to define a new quantum chromatic number and establish some basic properties of this number. We then suggest a way of studying these quantum graph homomorphisms using certain completely positive maps and describe their structure. Finally, we use these completely positive maps to define the notion of a ``quantum" core of a graph.   
\end{abstract}

\maketitle


\section{Introduction}
Let $G=(V(G),E(G))$ and $H= (V(H), E(H))$ be graphs on vertices $V(G)=
\{ 1,...,n \}$ and $V(H)= \{ 1,...,m \}$. The theory of  graph
homomorphisms is one of the central tools of graph theory and is used
in the development of the concept of the core of a graph. More
recently, work in quantum information theory has led to quantum
versions of many concepts in graph theory and there is an extensive
literature (\cite{CMRSSW}, \cite{dsw}, \cite{PSSTW}). In particular,
D. Roberson\cite{Ro} and L. Mancinska \cite{MR}
developed an extensive theory of quantum homomorphisms of graphs. D. Stahlke\cite{St} interpreted graph
homomorphisms in terms of ``completely positive(CP) maps on the traceless operator space of a graph''. 

These papers motivate us to consider quantum and classical graph
homomorphisms as special families of completely positive maps between
the operator systems of the graphs.

There is not just a
single quantum theory of graphs, but there are really possibly several
different quantum theories depending on the validity of certain
conjectures of Connes and Tsirelson. In earlier work on quantum
chromatic numbers\cite{PT, PSSTW}, we studied the differences and similarities between the properties of the
quantum chromatic numbers defined by the possibly different quantum theories.
We wish to parallel those ideas for quantum graph homomorphisms.
One technique of \cite{PSSTW} and \cite{DP} was to show that the existence of
quantum colorings was equivalent to the existence of certain types of
traces on a C*-algebra affiliated with the graph and we wish to expand
upon that topic here. This leads us to introduce the C*-algebra of a
graph homomorphism and we will show that the existence or
non-existence of various types of quantum graph homomorphisms are
related to properties of this C*-algebra, e.g., whether or not it has
any finite dimensional representations or has any traces.

  Finally, we wish to use our correspondence between quantum graph
  homomorphisms and CP maps to introduce a quantum analogue of the
  core of a graph.

\section{The Homomorphism Game}

Given graphs $G$ and $H$ a {\bf graph homomorphism} from $G$ to $H$ is
a mapping $f: V(G) \to V(H)$ such that 
\[(v,w) \in E(G) \implies (f(v), f(w)) \in E(H).\]
When a graph homomorphism from $G$ to $H$ exists we write $G \to
  H$.

Paralleling the work on quantum chromatic numbers \cite{PT}, we study a graph
homomorphism game, played by Alice, Bob, and a referee.  Given graphs
$G$ and $H$, the referee gives Alice and Bob a vertex of $G$, say $v$
and $w$, respectively, and they each respond with a vertex from $H$,
say $x$ and $y$, respectively.  Alice and Bob win provided that:
\[ v=w \implies x=y,\]
\[v \sim_G w \implies x \sim_H y.\]

If they have some random strategy and we let $p(x,y|v,w)$ denote the probability that we get outcomes $x$ and $y$ given inputs $v$ and $w$, then these equations translate as:
\begin{enumerate}
\item $p(x \ne y | v=w)=0$
\item $p(x\nsim_H y | v \sim_G w) =0 $
\end{enumerate}

Now say $G$ has $n$ vertices and $H$ has $m$ vertices. We consider the sets of correlations studied in \cite{PSSTW} and \cite{PT}:
\[ Q_l(n,m) \subseteq Q_q(n,m) \subseteq Q_{qa}(n,m) \subseteq Q_{qc}(n,m)\subseteq Q_{vect}(n,m).\]

In the appendix, we review the definition and some known facts about these sets.

For $t \in \{ l,q,qa,qc,vect \}$ we define:
\[ G \stackrel{t}{\longrightarrow} H,\]
provided that there exists
\[ p(x,y|v,w) \in Q_t(n,m) \]
satisfying (1) and (2). Notice that when we write $p(x,y|v,w) \in Q_t(n,m)$ we really mean $\big( p(x,y|v,w) \big)_{v,w,x,y} \in Q_t(n,m)$. Any $p(x,y|v,w) \in Q_t(n,m)$ satisfying these
conditions we call a {\bf winning t-strategy} and say that there
exists a {\bf quantum t-homomorphism} from $G$ to $H.$ 

The condition (1) is easily seen to be the {\bf synchronous
  condition} defined in \cite{PSSTW} and the subset of correlations
satisfying this condition was denoted $Q^s_t(n,m).$ Thus,
$p(x,y|v,w)$ is a winning t-strategy if and only if
$p(x,y|v,w) \in Q^s_t(n,m)$ and satisfies (2).
 
The following result is known, but we provide a
proof since we are using a slightly different (but equivalent)
characterization of $Q_l(n,m).$

\begin{thm}\label{c=l} Let $G$ and $H$ be graphs. Then $G \to H$ if and only if $G \stackrel{l}{\longrightarrow} H.$
\end{thm}
\begin{proof} First assume that $G \to H.$ Let $f: V(G) \to V(H)$ be a graph homomorphism. Let $\Omega = \{ t\}$ be the singleton probability space.  For each $v \in V(G)$ let Alice have the ``random variable'', $f_v(t) = f(v)$ and for each $w \in V(G)$ let Bob have the random variable $g_w(t) = f(w).$  Then 
\[p(x,y|v,w) := Prob(x=f_v(t), y= g_w(t)) = \begin{cases} 1, & \text{ when } x=f(v), y=f(w) \\ 0, & \text{else} \end{cases}.\]
From this it easily follows that $p(x,y|v,w)$ satisfies (1) and (2).

Conversely, assume that we have a probability space $(\Omega, P)$ and random variables $f_v, g_w: \Omega \to V(H)= \{1,...,m \}$ so that
$p(x,y|v,w) = P(x= f_v(\omega), y = g_w(\omega))$ satisfies (1) and (2).
By (1), for each $v$ the set $B_v = \{ \omega: f_v(\omega) = g_v(\omega) \}$ has probability 1.  Similarly, for each $(v,w) \in E(G)$ the set
$Q_{v,w}= \{ \omega: (f_v(\omega), g_w(\omega)) \in E(H) \}$ has probability 1. 
Thus,  
\[D = \big(\cap_{v \in V(G)} B_v \big) \cap \big(\cap_{(v,w) \in E(G)} Q_{v,w} \big)\] 
has measure 1, and so in particular is non-empty.
Fix any $\omega \in D$ and define $f: V(G) \to V(H)$ by  $f(v) := f_v(\omega)=g_v(\omega).$  Then whenever $(v,w) \in E(G)$ we have that $(f(v), f(w)) =(f_v(\omega), g_w(\omega)) \in E(H).$ Thus, $f$ is a graph homomorphism.
\end{proof}

Thus, quantum l-homomorphisms correspond to classical graph homomorphisms.

\begin{remark}\label{vect=V} In \cite{CMRSSW} several notions of graph
  homomorphisms were also introduced, including $G
  \stackrel{B}{\longrightarrow} H,$ $G \stackrel{V}{\longrightarrow}
  H$ and $G \stackrel{+}{\longrightarrow} H.$ A look at their
  definition shows that
$$G \stackrel{vect}{\longrightarrow} H \text{ if and only if } G \stackrel{V}{\longrightarrow} H$$
\end{remark}

\begin{cor}\label{arrows} Let $G$ and $H$ be graphs. Then
$$G \longrightarrow H \implies G\stackrel{q}{\longrightarrow} H \implies G\stackrel{qa}{\longrightarrow} H \implies G\stackrel{qc}{\longrightarrow} H \implies G\stackrel{vect}{\longrightarrow} H$$
\end{cor}
\begin{proof}
This is a direct consequence of the above definitions, Theorem
\ref{c=l}, and the corresponding set containments. 
\end{proof}

\section{Quantum Homomorphisms and CP Maps}
Recall \cite{OP} that the operator system of a graph $G$ on $n$ vertices is the subspace of the $n \times n$ complex matrices $M_n$ given by
\[ \cl S_G = span \{ E_{v,w}: v=w \text{ or } (v,w) \in E(G) \},\]
where $E_{v,w}$ denotes the $n \times n$ matrix that is 1 in the $(v,w)$-entry and
0 elsewhere.

We now wish to use a winning x-strategy for the homomorphism game to
define a CP map from $\cl S_G$ to $\cl S_H.$  It will suffice to do this in
the case of winning vect-strategies since every other strategy is a
subset.

\begin{prop}\label{phipdefn} Let $p(x,y|v,w) \in Q^s_{vect}(n,m)$, let $E_{v,w} \in M_n$ and $E_{x,y} \in M_m$
denote the canonical matrix unit bases. Then the linear map
$\phi_p:M_n \to M_m$ defined on the basis by
\[\phi_p(E_{v,w}) = \sum_{x,y} p(x,y|v,w)E_{x,y},\]
is completely positive.
\end{prop}
\begin{proof}By Choi's theorem \cite{vpbook}, to prove that $\phi_p$ is CP it is
  enough to prove that the Choi matrix,
\[P:= \sum_{v,w} E_{v,w} \otimes \phi_p(E_{v,w})= \sum_{v,w,x,y}
p(x,y|v,w)E_{v,w} \otimes E_{x,y} \in M_n \otimes M_m =
M_{nm}\]
is positive semidefinite.

Recall that by the definition and characterization of vector correlations satisfying the
 synchronous condition in \cite{PT} there exists a Hilbert space and vectors
  $\{h_{v,x} \}$ satisfying:
\begin{itemize}
\item $h_{v,x} \perp h_{v,y}$ for all $x \ne y$,
\item $\sum_x h_{v,x} = \sum_x h_{w,x}$ for all $v,w$,
\item $\|\sum_x h_{v,x} \|=1,$
\end{itemize}
 such that $p(x,y|v,w) = \langle h_{v,x}, h_{w,y} \rangle.$

Now let $\{e_v\}$ and $\{ f_x \}$ denote the canonical orthonormal
bases for $\bb C^n$ and $\bb C^m$, respectively, let $a_{v,x} \in \bb
C$ be arbitrary complex numbers, so that $k = \sum_{v,x} a_{v,x} e_v
\otimes f_x$ is an arbitrary vector in $\bb C^n \otimes \bb C^m.$ We
have that
\[ \langle Pk, k \rangle = \sum_{v,w,x,y} \overline{a_{v,x}} a_{w,y}
p(x,y|v,w) = \sum_{v,w,x,y} \overline{a_{v,x}} a_{w,y} \langle
h_{v,x}, h_{w,y} \rangle = \langle h,h \rangle, \]
 where $h = \sum_{v,x} \overline{a_{v,x}}h_{v,x}.$

Thus, $P$ is positive semidefinite and $\phi_p$ is CP.
\end{proof}

\begin{thm}\label{TP} Let $G$ and $H$ be graphs, let $p(x,y|v,w) \in Q^s_{vect}(n,m)$ be a
winning vect-strategy for a quantum vect-homomorphism from $G$ to $H$ and let
$\phi_p:M_n \to M_m$ be the CP map defined in Proposition~\ref{phipdefn}.
Then $\phi_p(\cl S_G) \subseteq \cl S_H$ and $\phi_p$ is trace-preserving on $\cl S_G.$
\end{thm}
\begin{proof}
To see that $\phi_p$ is trace preserving on $\cl S_G$ it will be enough to
show that $tr\big( \phi_p(E_{v,v}) \big) = tr(E_{v,v}) =1,$ and for $v
\sim_G w,$ $tr \big( \phi_p(E_{v,w}) \big) = tr(E_{v,w}) =0.$

When $v=w$ we have that
\[tr \big(\phi_p(E_{v,v})\big) = tr \big( \sum_{x,y} p(x,y|v,v)E_{x,y}
\big) = \sum_x
p(x,x|v,v) =1= tr(E_{v,v}),\] by the definition of $p$.

Finally, if $v \ne w$ and $E_{v,w} \in \cl S_G,$ then 
\[ tr \big( \phi_p(E_{v,w}) \big) = \sum_{x}p(x,x|v,w) =0=tr(E_{v,w}),\]
by (2) and the fact that $x \nsim_H x.$

Hence, $\phi_p$ is trace-preserving on $\cl S_G.$

Now we prove that $\phi(\cl S_G) \subseteq \cl S_H.$ First, $\phi_p(E_{v,v}) =
\sum_{x,y} p(x,y|v,v)E_{x,y},$ but since $p$ is synchronous,
$p(x,y|v,v) =0$ for $x \ne y.$ Hence, $\phi_p(E_{v,v})$ is a diagonal
matrix and so in $\cl S_H.$ To finish the proof it will be enough to show
that when $v \sim_G w,$ we have $\phi_p(E_{v,w}) \in \cl S_H.$ But by
property (2), $p(x,y|v,w) =0$ when $x \nsim_H y.$ Thus,
$\phi_p(E_{v,w}) \in \cl S_H.$ In fact, it is a matrix with 0-diagonal in $\cl S_H.$
\end{proof}

\begin{cor} Let $x \in \{ l,q,qa,qc, vect \}$. If $p(x,y|v,w) \in Q_x^s(n,m)$ is a
  winning x-strategy, then the map $\phi_p: M_n \to M_m$ is CP,
  $\phi_p(\cl S_G) \subseteq \cl S_H$ and $\phi_p$ is
  trace-preserving on $\cl S_G.$ We say that the correlation $p(x,y|v,w)$ {\bf implements} the quantum x-homomorphism.
\end{cor}

\begin{exam}\label{classicisloc} Suppose we have a graph homomorphism $G \to H$ given by $f:V(G) \to V(H).$  If we let $\Omega=\{t\}$ be a one point probability space and define Alice and Bob's random variables $f_v, g_w: \Omega \to V(H)$ by $f_v(t) = f(v), g_w(t)= f(w)$ as in the proof of Theorem~\ref{c=l}, then we obtain $ p(x,y|v,w) \in Q^s_{l}(n,m)$ with
\[  p(x,y|v,w) = Prob( f_v=x, g_w=y) = \begin{cases} 1 & x=f(v), y= f(w) \\ 0 & \text{ else } \end{cases}.\]
The corresponding CP map satisfies
\[ \phi_p(E_{v,w}) = E_{f(v),f(w)}.\]

\end{exam}

We now wish to turn our attention to the composition of
quantum graph homomorphisms.
First we need a preliminary result.

\begin{prop} Let $x \in \{ l, q, qa, qc, vect \}$, let $p(x,y|v,w) \in
  Q_x(n,m)$ and let $q(a,b|x,y) \in
  Q_x(m,l)$. Then
\[ r(a,b|v,w) := \sum_{x,y} q(a,b|x,y)p(x,y|v,w) \in Q_x(n,l).\]
Moreover, if $p$ and $q$ are synchronous, then $r$ is synchronous.
\end{prop}
\begin{proof} First we show the synchronous condition is met by $r$.
  Suppose that $v=w$ and $a \ne b.$  Since $p$ is synchronous, all the
  terms $p(x,y|v,v)$ vanish unless $x=y.$  Thus,  $r(a,b|v,v) = \sum_x
  q(a,b|x,x)p(x,x|v,v).$ But because $q$ is synchronous, each
  $q(a,b|x,x) =0.$  Hence, if $a \ne b,$ then $r(a,b|v,v) =0.$

The cases when $x=l, q, qa, qc$ are shown in \cite[Lemma 6.5]{PSSTW}

Finally we tackle the case when $x= vect.$ In this case, we are given Hilbert spaces $\cl H_1, \cl H_2$, unit vectors $\eta_1 \in \cl H_1,
\eta_2 \in \cl H_2$, and vectors $h_{v,x}, k_{w,y} \in
\cl H_1$, $f_{x,a}, g_{y,b} \in \cl H_2$ such that:
\begin{multline*}
h_{v,x} \perp h_{v,y}, k_{v,x} \perp k_{v,y}, \forall x \ne y, \,\,\,
f_{x,a} \perp f_{x,b}, g_{x,a} \perp g_{x,b}, \forall a \ne b, \\
\sum_x h_{v,x} = \sum_x k_{v,x} = \eta_1, \forall v, \,\,\, \sum_a
f_{x,a} = \sum_a g_{x,a} = \eta_2, \forall x
\end{multline*}
such that $p(x,y|v,w) = \langle h_{v,x}, k_{w,y} \rangle$ and
$q(a,b|x,y) = \langle f_{x,a}, g_{y,b} \rangle.$

We set $\alpha_{v,a} = \sum_x h_{v,x} \otimes f_{x,a}$ and
$\beta_{w,b} = \sum_y k_{w,y} \otimes g_{y,b}.$ Now one checks that
these vectors satisfy all the necessary conditions, e.g.,
$\alpha_{v,a} \perp \alpha_{v,b}, \, \forall a \ne b$ and $\sum_a
\alpha_{v,a} = \eta_1 \otimes \eta_2, \forall v,$ and that
\[ \langle \alpha_{v,a}, \beta_{w,b} \rangle = \sum_{x,y} \langle
h_{v,x}, k_{w,y} \rangle \langle f_{x,a}, g_{y,b} \rangle = r(a,b|x,y).\] 
\end{proof}

\begin{cor}\label{comp} Let $x \in \{ l,q,qa,qc, vect \}$, let $p(x,y|v,w) \in
  Q_x(n,m),$ $q(a,b|x,y) \in Q_x(m,l)$ and let $r(a,b|v,w) =
  \sum_{x,y} q(a,b|x,y)p(x,y|v,w) \in Q_x(n,l).$ If $\phi_p:M_n \to
  M_m$, $\phi_q:M_m \to M_l$ and $\phi_r:M_n \to M_l$ are the
  corresponding linear maps, then $\phi_r = \phi_q \circ \phi_p.$
\end{cor}

The following is now immediate:

\begin{thm}\label{compgraphhomo} Let $x \in \{ l,q,qa,qc,vect \}$, let $G, H$ and $K$ be
  graphs on $n,m$ and $l$ vertices, respectively, and assume that
$G \stackrel{x}{\rightarrow} H$, $H \stackrel{x}{\rightarrow} K$. If $p(x,y|v,w)
  \in Q_x(n,m)$ and $q(a,b|x,y) \in Q_x(m,l)$ are winning quantum
  x-strategies for homomorphisms from $G$ to $H$ and $H$ to $K$, respectively,  then $r(a,b|v,w)= \sum_{x,y} q(a,b|x,y)p(x,y|v,w) \in Q_x(n,l)$ is a winning
  x-strategy for a homomorphism from $G$ and $K$, so that $G \stackrel{x}{\rightarrow} K.$ In summary, 
  $$\text{if } G \stackrel{x}{\rightarrow} H \text{ and } H \stackrel{x}{\rightarrow} K \text{, then } G \stackrel{x}{\rightarrow} K.$$
\end{thm}

\section{C*-algebras and Graph Homomorphisms}

We wish to define a C*-algebra $\cl A(G,H)$
generated by the relations arising from a winning strategy for the
graph homomorphism game.

\begin{defn}\label{GHalg} Let $G$ and $H$ be graphs. A set of
  projections $\{ E_{v,x}: v \in V(G), \, x \in V(H) \}$ on a Hilbert
  space $\cl H$ satisfying the following relations:
\begin{enumerate}
\item for each $v \in V(G),$ $ \sum_x E_{v,x}= I_{\cl H},$
\item if $(v,w) \in E(G)$ and $(x,y) \notin E(H)$ then $E_{v,x} E_{w,y} =0,$
\end{enumerate}
is called a {\bf representation of the graph homomorphism game from
$G$ to $H$}. If no
set of projections on any Hilbert space exists satisfying these
relations, then we say that the {\bf graph homomorphism game from $G$
  to $H$ is not
  representable}.
\end{defn}

\begin{defn}
Let $G$ and $H$ be graphs. If a representation of the graph homomorphism game exists, then we let
$\cl A(G,H)$ denote the ``universal'' C*-algebra generated by such
sets of projections. If the graph homomorphism game from $G$ to $H$ is not representable, then we
say that {\bf $\cl A(G,H)$ does not exist}. We write $G \stackrel{C^*}{\longrightarrow} H$ if and only if $\cl
A(G,H)$ exists.
\end{defn}
By ``universal'' we mean that $\cl A(G,H)$ is a unital C*-algebra
generated by projections $\{ e_{v,x}: v \in V(G), x \in V(H) \}$
satisfying
\begin{enumerate}
\item for each $v \in V(G),$ $\sum_x e_{v,x} =1,$
\item if $(v,w) \in E(G)$ and $(x,y) \notin E(H),$ then
  $e_{v,x}e_{w,y} =0,$
\end{enumerate}
with
the property that for
any representation of the graph homomorphism game on a Hilbert space $\cl H$ by projections $\{ E_{v,x}
\}$
satisfying the above relations, there exists a *-homomorphism
$\pi: \cl A(G,H) \to B(\cl H)$ with $\pi(e_{v,x}) = E_{v,x}.$

Here is one result that relates to existence.  Let $E_m$ be the
``empty'' graph on $m$ vertices, i.e., the graph with no edges.

\begin{prop} Let $G$ be a graph with at least one edge, $(v,w) \in
  E(G)$. Then $\cl A(G,E_m)$ does not exist.
\end{prop}
\begin{proof} By definition we have that $e_{v,x}e_{w,y} =0$ for all
  $x,y.$ Thus,
\[0 = \sum_{x,y} e_{v,x}e_{w,y} = \big( \sum_x e_{v,x} \big) \big(
\sum_y e_{w,y} \big) = 1,\]
a contradiction.
\end{proof}

In \cite[Definition 2]{CMRSSW} another type of graph homomorphism was defined, denoted by $G \stackrel{B}{\to} H$. Briefly, if in our definition of $Q_{vect}(n,m)$ we had dropped the requirement that all the inner products be non-negative, then we would obtain a larger set of tuples and $G \stackrel{B}{\to} H$ if and only if there exists a $p(x,y|v,w)$ in this larger set satisfying the conditions (1) and (2) of a winning strategy for the graph homomorphism game.  Note that in this case, since these numbers need not be non-negative, we cannot interpret them as probabilities.

\begin{prop}\label{C*impliesB} If $G
  \stackrel{C^*}{\longrightarrow} H$ or $G \stackrel{vect}{\to} H$, then
  $G \stackrel{B}{\longrightarrow} H$, as defined in \cite{CMRSSW}.  
\end{prop}
\begin{proof} The vect case is obvious from the remarks above.  Let $\{ E_{v,x}: v \in V(G), x \in V(H) \}$ be a set of
  projections that yields a
  representation of the graph homomorphism game on a Hilbert space $\cl H$ and let $h
  \in \cl H$ be any unit vector.

If we set $h^v_x = E_{v,x}h,$ then set of vectors $\{ h^v_x\}$
satisfies all the properties of the definition of $G
\stackrel{B}{\longrightarrow} H$ in \cite[Definition 2]{CMRSSW}.
\end{proof}

\begin{remark} We do not know necessary and sufficient conditions for
  $\cl A(G,H)$ to exist. In particular, we do not know if $G
  \stackrel{B}{\to} H$ implies $G
  \stackrel{C^*}{\longrightarrow}H$.
\end{remark}

\begin{prop}
If $G \stackrel{C^*}{\to } H$ and $H\stackrel{C^*}{\to } K$, then $G\stackrel{C^*}{\to } K$.
\end{prop}
\begin{proof}
Since $G \stackrel{C^*}{\to } H$ and $H\stackrel{C^*}{\to } K$, then
we know that there exist projections $\{ E_{v,x} \}$ and $\{F_{y,a}\}$
with $v\in V(G)$, $x,y\in V(H)$ and $a\in V(K)$ on Hilbert spaces $\cl
H$ and $\cl K$, respectively, satisfying $(1)$ and $(2)$. Consider the
set of self-adjoint operators on $\cl H \otimes \cl K$ defined by $G_{v,a}= \sum_{x\in V(H)} E_{v,x} \otimes F_{x,a} $ for $x\in V(G)$ and $a\in V(K)$. Notice that,
\begin{multline*}
G_{v,a}G_{v,a} =(\sum_{x} E_{v,x} \otimes F_{x,a})(\sum_{y} E_{v,y}
\otimes F_{y,a})= \\ \sum_{x,y} E_{v,x}E_{v,y} \otimes F_{x,a}F_{y,a}=
\sum_{x} E_{v,x} \otimes F_{x,a}= G_{v,a} 
\end{multline*}
by (2) and the fact that $E_{v,x}$ and $F_{x,a}$ are
projections. Thus, each $G_{v,a}$ is a projection. Furthermore, for each $v\in V(G)$,
$$\sum_a G_{v,a} = \sum_a \sum_x E_{v,x}\otimes F_{x,a}= \sum_x E_{v,x}\otimes (\sum_a F_{x,a}) = (\sum_x E_{v,x})\otimes I_{\cl K} = I_{\cl H}\otimes I_{\cl K} $$
by (1).
Moreover, for each $(v,w)\in E(G)$ and $(a,b)\not \in E(K)$,
$$G_{v,a}G_{w,b} = (\sum_x E_{v,x}\otimes F_{x,a})(\sum_{y} E_{w,y} \otimes F_{y,b})
=\sum_x \sum_y (E_{v,x}\otimes F_{x,a})(E_{w,y} \otimes F_{y,b})$$
$$= \sum_x \sum_y E_{v,x}E_{w,y}\otimes F_{x,a}F_{y,b} 
= \sum_{x\sim y} E_{v,x}E_{w,y}\otimes F_{x,a}F_{y,b} = 0$$
by (2). Hence, $\{ G_{v,a} : v \in V(G), a \in V(K) \}$ is a
representation of a graph homomorphism game from $G$ to $K$.
\end{proof}

Recall that a {\bf trace} on a unital C*-algebra $\cl B$ is any state $\tau$ such that  
$\tau(ab) = \tau(ba)$ for all $a,b \in \cl B.$ 


\begin{thm}\label{trhomo} Let $G$ be a graph and let $x \in \{ l, q,
  qa, qc, vect \}$.
\begin{enumerate}
\item $G \stackrel{qc}{\rightarrow} H$ if and only if there exists a
  tracial state on $\cl A(G,H),$
\item if $ G \stackrel{qc}{\rightarrow} H$, then $G \stackrel{C^*}{\rightarrow} H,$
\item $G \stackrel{q}{\rightarrow} H$ if and only if $\cl A(G,H)$ has a finite dimensional representation,
\item $G \rightarrow H$ if and only if $\cl A(G,H)$ has an abelian representation.
\end{enumerate}
\end{thm}
\begin{proof} We have that $G \stackrel{qc}{\rightarrow} H$ if and
  only if there exists a winning $qc$-strategy $p(x,y|v,w) \in
  Q^s_{qc}(n,m)$. By \cite{PSSTW} this strategy must be given by a
  trace on the algebra generated by Alice's operators with $p(x,y|v,w)
  = \tau(A_{v,x}A_{w,y}).$  Moreover, in the GNS representation, this
  trace will be faithful.

We now wish to show that these operators satisfy the necessary
relations to induce a representation of $\cl A(G,H).$

By the original hypotheses, we will have that $A_{v,x}A_{v,y} =0$ for $x \ne y.$
When $(v,w) \in E(G)$ and $(x,y) \not \in E(H),$ we will have that
$\tau(A_{v,x}A_{w,y}) = p(x,y|v,w) =0$ and hence, $A_{v,x}A_{w,y}=0.$

Thus, Alice's operators give rise to a representation of $\cl A(G,H)$ and composing this *-homomorphism with the tracial state on the algebra generated by Alice's operators gives the trace on $\cl A(G,H).$ The converse follows by setting $p(x,y|v,w)=\tau(A_{v,x} A_{w,y}).$

Clearly, (2) follows from (1).

The proof of (3) is similar to the proof of (1). In this case since $p(x,y|v,w) \in
Q^s_q(n,m)$ the operators all live on a finite dimensional space and
hence generate a finite dimensional representation. 

The proof of (4) first uses the fact that $G \rightarrow H$ if and
only if $G \stackrel{l}{\rightarrow} H$ (\ref{c=l}). If we let $(\Omega, \lambda)$
be the corresponding probability space and let $f_v, g_w: \Omega \to
V(H)$ be the random variables for Alice and Bob, respectively, then
the conditions imply that $f_v = g_v$ a.e. If we let $E_{v,x}$ denote
the characteristic function of the set $f^{-1}(\{ x\}),$ then it is
easily checked that these projections in $L^{\infty}(\Omega, \lambda)$
satisfy all the conditions needed to give an abelian representation of
$\cl A(G,H).$ 
\end{proof}

Note that saying that $\cl A(G,H)$ has an abelian representation is
equivalent to requiring that it has a one-dimensional representation.

We now apply these results to coloring numbers.
Let $K_c$ denote the complete graph on $c$ vertices.

\begin{prop} Let $x \in \{ l, q, qa, qc, vect \}$,  then $\chi_x(G)$ is the least integer $c$ for which $G \stackrel{x}{\rightarrow} K_c.$
\end{prop}
\begin{proof} Any winning $x$-strategy for a homomorphism from $G$ to $H$ is a winning strategy for a $x$-coloring.
\end{proof}

The above result motivates the following definition.

\begin{defn} Define $\chi_{C^*}(G)$ to be
  the least integer $c$ for which $G \stackrel{C^*}{\rightarrow} K_c.$ Similarly, define $\omega_{C^*}(G)$ to be the biggest integer $c$ for which $K_c \stackrel{C^*}{\rightarrow} G.$
\end{defn}

We let $\vartheta(G)$ denote the Lovasz theta function of a graph $G$
and we let $\overline{G}$ denote the graph with the same vertex set as
$G$ and edges defined by $(v,w) \in E(\overline{G}) \iff v \ne w
\text{ and } (v,w) \notin E(G)$.

\begin{prop} Let $G$ be a graph, then
$$\omega_{C^*}(G)\leq \vartheta(\overline{G})\leq \chi_{C^*}(G).$$
\end{prop}
\begin{proof}
Let $c:=\chi_{C^*}(G)$. If we combine \ref{C*impliesB} with \cite[Theorem 6]{CMRSSW} we know that
$$G \stackrel{C^*}{\rightarrow} K_c \implies G \stackrel{B}{\rightarrow} K_c \iff \vartheta(\overline{G})\leq \vartheta(\overline{K_n})=c.$$ Similarly, if you apply the above proof to $K_d \stackrel{C^*}{\rightarrow} G$, where $d:=\omega_{C^*}(G)$,  you get the remaining inequality. 
\end{proof}

\begin{remark} Since $G \stackrel{qc}{\rightarrow} K_c \implies G
  \stackrel{C^*}{\rightarrow} K_c$, we have that
$\chi_{qc}(G) \ge \chi_{C^*}(G),$ but we don't know the relation
between $\chi_{C^*}(G)$ and $\chi_{vect}(G).$
\end{remark}

This leads to the following results:

\begin{thm}Let $G$ be a graph.
\begin{enumerate}
\item $\chi(G)$ is the least integer $c$ for which there is an abelian representation of $\cl A(G, K_c).$
\item $\chi_q(G)$ is the least integer $c$ for which $\cl A(G, K_c)$ has a finite dimensional representation.
\item $\chi_{qc}(G)$ is the least integer $c$ for which $\cl A(G,
  K_c)$ has a tracial state.
\item $\chi_{C^*}(G)$ is the least integer $c$ for which $\cl A(G,
  K_c)$ exists.
\end{enumerate}
\end{thm}

\begin{thm} Let $G$ be a graph.
\begin{enumerate}
\item The problem of determining if $\cl
  A(G,K_3)$ has an abelian representation is
  NP-complete.
\item The problem of determining if
  $\cl A(G, K_3)$ has a finite dimensional representation is
  NP-hard.
\item The problem of determining if $\cl A(G,K_c)$ has a trace is
  solvable by a semidefinite programming problem. 
\end{enumerate}
\end{thm}
\begin{proof} We have shown that $\cl A(G,K_3)$ has an abelian representation if and only
  if $G$ has a 3-coloring and this latter problem is NP-complete \cite{Da}.

In \cite[Theorem 1]{Ji}, it is proven that an NP-complete problem is
polynomially reducible to determining if $\chi_q(G)=3$. Hence, this
latter problem is NP-hard. 

In \cite{PSSTW}, it is proven that for each $n$ and $c$ there is a
spectrahedron $S_{n,c} \subseteq \bb R^{n^2c^2}$ such that for each graph $G$ on $n$ vertices
there is a linear functional $L_G : \bb R^{n^2c^2} \to \bb R$
with the property that $\chi_{qc}(G) \le c$ if and only if there is a
point $p \in S_{n,c}$ with $L_G(p) =0.$ Thus, determining if
$\chi_{qc}(G) \le c$ is solvable by a semidefinite programming
problem. But we have seen that $\chi_{qc}(G) \le c$ if and only if
$\cl A(G, K_c)$ has a trace.
\end{proof}

\begin{remark}
Currently, there are no known algorithms for
  determining if $\chi_q(G) \le 3,$ i.e., for determining if $\cl
  A(G,K_3)$ has a finite dimensional representation.
\end{remark}

\begin{remark} We do not know the complexity level of determining if
  $\cl A(G,H)$ exists. In particular, we do not know the complexity
  level of determining if $G \stackrel{C^*}{\rightarrow} K_3,$ or any
  algorithm.
\end{remark}

\begin{remark} In \cite{CMRSSW} it is proven that $\chi_{vect}(G) =
  \lceil \vartheta^+( \overline{G}) \rceil,$ which is solvable by an
  SDP.
\end{remark}

\begin{remark} There is a family of finite input, finite output games
  that are called {\bf synchronous games}\cite{DP}, of which the graph
  homomorphism game is a special case. For any synchronous game $\cl
  G$ we can construct the C$^*$-algebra of the game $\cl A(\cl G)$ and
  there are analogues of many of the above theorems. For instance, the
  game will have a winning qc-strategy, q-strategy or l-strategy if
  and only if $\cl A(\cl G)$ has a trace, finite dimensional, or
  abelian representation, respectively.
\end{remark}

\section{Factorization of Graph Homomorphisms}

In this section, we show that the CP maps that arise from graph
homomorphisms have a canonical factorization involving $\cl A(G,H).$

\begin{prop} Let $G$ and $H$ be graphs on $n$ and $m$ vertices, respectively. The map $\Gamma: M_n \to M_m(\cl A(G,H))$ defined on matrix units by $\Gamma(E_{v,w}) = \sum_{x,y} E_{x,y} \otimes e_{v,x}e_{w,y}$ is CP.
\end{prop}
\begin{proof} Let $E_{v,x}, v \in V(G), x \in V(H)$ denote the $n \times m$ matrix units.  Let $Z= \sum_{w,y} E_{w,y} \otimes e_{w,y} \in M_{n,m}(\cl A(G,H)).$ Then
\[\Gamma( \sum_{v,w} c_{v,w} E_{v,w}) = Z^* \big( c_{v,w}E_{v,w} \otimes I \big) Z,\]
where $I$ denotes the identity of $\cl A(G,H)$ and $\big( c_{v,w}E_{v,w} \otimes I \big) \in M_n( \cl A(G,H)).$
\end{proof}

Let $p(x,y|v,w) \in Q^s_{qc}(n,m)$ be a winning $qc$-strategy for a graph homomorphism from $G$ to $H$. Then there is a tracial state $\tau: \cl A(G,H) \to \bb C$ such that  $p(x,y|v,w) = \tau(e_{v,x}e_{w,y})$ and hence $\phi_p$ factors as $\phi_p = (id_m \otimes \tau) \circ \Gamma,$ where $id_m \otimes \tau: M_m(\cl A(G,H)) \to M_m.$ Conversely, if $\tau: \cl A(G,H) \to \mathbb{C}$ is any tracial state, then $(id_m \otimes \tau) \circ \Gamma = \phi_p$ for some winning $qc$-strategy $p(x,y|v,w) \in Q^s_{qc}(n,m).$

Similarly, this map $\phi_p$ arises from a winning $q$-strategy if and only if it arises from a $\tau$ that has a finite dimensional GNS representation and from a winning $l$-strategy if and only if it arises from a $\tau$ with an abelian  GNS representation.


This factorization leads to the following result. Recall that $\vartheta(G)$ denotes the Lovasz theta function of a graph and let $\|\phi\|_{cb}$ denote the completely bounded norm of a map.

\begin{lemma} Let $G$ be a graph on $n$ vertices, let $\cl H$ be a Hilbert space, let $P_{v,w} \in B(\cl H), \, \forall v,w \in V(G)$ and regard $P=(P_{v,w})$ as an operator on $\cl H \otimes \bb C^n.$ If
\begin{enumerate}
\item $P=(P_{v,w}) \ge 0$,
\item $P_{v,v} =I_{\cl H},$
\item $(v,w) \in E(G) \implies P_{v,w} =0,$
\end{enumerate}
then $\|P\| \le \vartheta(G).$
\end{lemma}
\begin{proof} Any vector $k \in \cl H \otimes \bb C^n$ has a unique representation as $k = \sum_v k_v \otimes e_v,$ where $k_v \in \cl H$ and $e_v \in \bb C^n$ denotes the standard orthonormal basis. Set $h_v = k_v/\|k_v\|$ (with $h_v=0$ when $k_v=0$), and $\lambda_v = \|k_v\|.$ Let $y= \sum_v \lambda_v e_v \in \bb C^n$ so that $\|y\|_{\bb C^n} = \|k\|.$ 
Set $B_k=\big( \langle P_{v,w}h_w, h_v \rangle \big) \in M_n = B(\bb C^n),$ so that
$\langle P k, k \rangle_{\cl H \otimes \bb C^n} = \langle B_k y,y \rangle_{\bb C^n}.$

This observation shows that if for any $h_v \in \cl H, \, \forall v \in V(G)$ with $\|h_v\|=1$ we let $\big( \langle P_{v,w}h_w, h_v \rangle \big) \in M_n = B(\bb C^n),$
then
\[ \|P\|= \sup \{ \|( \langle P_{v,w} h_w, h_v \rangle ) \|_{M_n} : \|h_v\|= 1 \}.\]

Now by the above hypotheses each matrix $(\langle P_{v,w}h_w, h_v \rangle)\ge 0,$ has all diagonal entries equal to 1 and  $(v,w) \in E(G) \implies \langle P_{v,w}h_w,h_v \rangle =0.$ Thus, by \cite{Lo},  $\|(\langle P_{v,w}h_w,h_v \rangle) \| \le \vartheta(G).$
\end{proof}

\begin{prop} Let $p(x,y|v,w) \in Q^s_{qc}(n,m)$ be a winning
  $qc$-strategy for a graph homomorphism from $G$ to $H.$  Then $\|\phi_p\|_{cb} \le \vartheta(G).$
\end{prop}
\begin{proof} Since $id_m \otimes \tau$ is a completely contractive map, we have that
$\|\phi_p\|_{cb} \le \|\Gamma\|_{cb}.$ Since this map is CP, by \cite{vpbook} we have that
\[ \|\Gamma\|_{cb} = \|\Gamma(I) \|= \|Z^*Z\|=\|ZZ^*\|. \]
Since $e_{w,y}^* = e_{w,y}$, we have
\[ ZZ^* = \sum_{v,w,x,y} (E_{v,x} \otimes e_{v,x})(E_{w,y} \otimes e_{w,y})^* =
\sum_{v,w} E_{v,w} \otimes \big(\sum_x e_{v,x}e_{w,x} \big).\]

Now if we let $p_{v,w}$ denote the $(v,w)$-entry of the above matrix in $M_n(\cl A(G,H)),$ then $p_{v,v}= \sum_x e_{v,x} = I.$ When $(v,w) \in E(G),$ then by Definition~\ref{GHalg}(3), we have that $p_{v,w}=0.$

Hence, by the above lemma, $\|ZZ^*\|\le \vartheta(G).$
\end{proof}

\section{Quantum Cores of Graphs}

A {\bf retract} of a graph $G$ is a subgraph $H$ of $G$ such that
there exists a graph homomorphism $f: G\to H$, called a 
{\bf retraction} with $f(x)=x$ for any $x\in V(H)$. A {\bf core} is a graph which does not retract to a proper subgraph \cite{homo}. 

Note that if $f:G \to G$ is an idempotent graph homomorphism and we
define a graph $H$ by setting $V(H) = f(V(G))$ and defining $(x,y) \in
E(H)$ if and only if there exists $(v,w) \in E(G)$ with $f(v) =x, f(w)
=y,$ then $H$ is a subgraph of $G$ and $f$ is a retraction onto $H.$
We denote $H$ by $f(G).$

The following result is central to proofs of the existence of cores of graphs.

\begin{thm}[\cite{homo}]\label{retraction}
Let $f$ be an endomorphism of a graph $G$. Then there is an $n$
such that $f^n$ is idempotent and a retraction onto $R = f^n(G)$.
\end{thm}

Our goal in this section is to attempt to define a quantum analogue of the core using completely positive maps, in particular we will use the above theorem as a guiding principle.

For $A=(a_{ij})\in M_n$, denote $||A||_1=\sum_{i,j} |a_{ij}|$ and $\sigma(A)=\sum_{i,j} a_{ij}$. Let $\phi_p: M_n\to M_m$, $\phi_p(E_{vw})=\sum_{x,y}p(x,y|v,w)E_{xy}$, for some $p(x,y|v,w) \in Q^s_{vect}(n,m)$.  Before we continue our discussions on cores we will need the following facts:

\begin{lemma}
$$\sigma(\phi_p(A))=\sigma(A)$$
\end{lemma}
\begin{proof}
By linearity it is enough to show the claim for matrix units,
$$\sigma(\phi_p(E_{vw}))=\sum_{x,y} p(x,y|v,w)=\sum_{x,y} \langle h_{v,x}, h_{w,y} \rangle=$$
$$\langle \sum_x h_{v,x}, \sum_y h_{w,y} \rangle=\langle \eta, \eta \rangle=1=\sigma(E_{vw})$$
\end{proof}


\begin{lemma}\label{1-norm}
Let $A=(a_{vw})$ be a matrix, then $$||\phi_p(A)||_{1}\le ||A||_{1}$$ 
If the entries of $A$ are non-negative, then $\|\phi_p(A)\|_1 = \|A\|_1$. 
\end{lemma}
\begin{proof}
We have
$$||\phi_p(A)||_{1}=\sum_{x,y}|\sum_{v,w} p(x,y|v,w)a_{v,w}| \le
\sum_{v,w} |a_{v,w}|(\sum_{x,y} p(x,y|v,w))$$ 
$$=\sum_{v,w} |a_{vw}|=||A||_{1}$$
When the entries of $A$ are all non-negative, the first inequality is
an equality. \end{proof}

For the next step in our construction we need to recall the concept of
a {\it Banach generalized limit}. A Banach generalized limit is a positive
linear functional $f$ on $\ell^{\infty}(\bb N)$, such that:
\begin{itemize}
\item if $(a_k) \in \ell^{\infty}(\bb N)$ and $\lim_k a_k$ exists,
  then $f((a_k)) = \lim_k a_k,$
\item if $b_k= a_{k+1}$, then $f((b_k))=f((a_k)).$
\end{itemize}
The existence and construction of these are presented in \cite{conway},
along with many of their other properties. Often a Banach generalized limit functional is written as $glim$.

Now fix a Banach generalized limit $glim$, assume that $n=m$, and that $\phi_p: M_n\to M_n$,
$\phi_p(E_{vw})=\sum_{x,y}p(x,y|v,w)E_{xy}$, for some $p(x,y|v,w) \in
Q^s_{qc}(n,n)$. Fix a matrix $A \in M_n$ and set
\[ a_{x,y}(k) = \langle \phi_p^k(A)e_y, e_x \rangle \]
so that $\phi_p^k(A) = \sum_{x,y} a_{x,y}(k) E_{x,y}.$
By Lemma~\ref{1-norm}, for every pair, $(x,y)$ the sequence
$(a_{x,y}(k)) \in \ell^{\infty}(\bb N).$

We define a map, $\psi_p: M_n \to M_n$ by setting
\[\psi_p(A) = \sum_{x,y} glim((a_{x,y}(k))) E_{x,y}.\]
Alternatively, we can write this as
\[ \psi_p(A) = (id_n \otimes glim)\phi^k_p(A).\]

 \begin{prop} Let $\big( p(x,y|v,w) \big) \in Q^s_{vect}(n,n)$ and let $\psi_p: M_n \to M_n$ be the map obtained as above via some Banach generalized limit, $glim.$ Then:
\begin{enumerate}
\item $\psi_p$ is CP,
\item $\sigma(\psi_p(A))= \sigma(A)$ for all $A \in M_n,$
\item $\|\psi_p(A)\|_1 \le \|A\|_1,$
\item $\psi_p \circ \phi_p = \phi_p \circ \psi_p= \psi_p,$
\item $\psi_p \circ \psi_p = \psi_p.$
\end{enumerate}
\end{prop}
\begin{proof} The first two properties follow from the linearity of the glim functional. For example, if $A= (a_{x,y})$ and $h= (h_1,...,h_n) \in \bb C^n,$ then
 $$ \langle \psi_p(A) h, h \rangle = \sum_{x,y} glim((a_{x,y}(k)))h_y\overline{h_x} = glim \big( \sum_{x,y} a_{x,y}(k)h_y \overline{h_x} \big) $$
 $$ = glim \big(\langle \phi_p^k(A)h,h \rangle \big)$$ 
If $A \ge 0$, then $\phi^k(A) \ge 0$ for all $k,$ and so is the above function of $k$. Since $glim$ is a positive linear functional, we find $A \ge 0$ implies $\langle \psi_p(A)h,h \rangle \ge 0,$ for all $h.$  This shows that $\psi_p$ is a positive map. The proof that it is CP is similar, as is the proof that it preserves $\sigma.$

The proof of the third property is similar to the proof of Lemma~\ref{1-norm}.

For the next claim, we have that 
\[\psi_p(\phi_p(A)) = (id \otimes glim)(\phi_p^{k+1}(A)) = (id \otimes glim)(\phi_p^k(A)) = \psi_p(A).\]
If we set $\psi_p(A) = \sum_{v,w} b_{v,w} E_{v,w},$ with $b_{v,w}= glim ( a_{v,w}(k)),$
then

$$\phi_p(\psi_p(A)) = \sum_{x,y,v,w} p(x,y|v,w) b_{v,w} E_{x,y} $$
$$=\sum_{x,y}  glim \big( \sum_{v,w} p(x,y|v,w) a_{v,w}(k) \big) E_{x,y}
 = \sum_{x,y} glim \big( a_{x,y}(k+1) \big) E_{x,y} = \psi_p(A)$$

Finally, to see the last claim, we have that
\[ \psi_p( \psi_p(A)) = (id \otimes glim)( \phi_p^k(\psi_p(A))) = (id \otimes glim)(\psi_p(A)) = \psi_p(A),\]
since the $glim$ of a constant sequence is equal to the constant.
\end{proof}

\begin{thm} Let $G$ be a graph on $n$ vertices, let $x \in \{ l, qa, qc, vect \}$ and let $ p(x,y|v,w) \in Q^s_x(n,n)$ be a winning $x$-strategy implementing a quantum graph $x$-homomorphism from $G$ to $G$.  Set $p_1(x,y|v,w) = p(x,y|v,w)$ and recursively define
\[ p_{k+1}(x,y|v,w) = \sum_{a,b} p(x,y|a,b) p_k(a,b|v,w).\]
If we set $r(x,y|v,w) = glim \big( p_k(x,y|v,w) \big),$ then $r(x,y|v,w)  \in Q^s_x(n,n)$ is a winning $x$-strategy implementing a graph $x$-homomorphism from $G$ to $G$ such that:
\begin{enumerate}
\item $\psi_p = \phi_r,$
\item $r(x,y|v,w) = \sum_{a,b} r(x,y|a,b) r(a,b|v,w).$
\end{enumerate}
\end{thm}
\begin{proof} By Theorem~\ref{compgraphhomo}, $\phi_p^k = \phi_{p_k},$ and $p_k$ is a winning $x$-strategy for a graph $x$-homomorphism from $G$ to $G$.  Thus,
\begin{multline*} \psi_p(E_{v,w}) = (id \otimes glim)(\phi_p^k(E_{v,w})) = (id \otimes glim)( \phi_{p_k}(E_{v,w}))\\ = \sum_{x,y} glim \big( p_k(x,y|v,w) \big) E_{x,y} = \phi_r(E_{v,w}).\end{multline*}
Thus, (1) follows.

Since $\phi_r \circ \phi_r = \psi_p \circ \psi_p = \psi_p = \phi_r,$ $(2)$ follows from Proposition~\ref{comp}.

Finally, if a bounded sequence of matrices $A_k= \big( a_{v,w}(k)\big) \in M_n$ all belong to a closed set, then it is not hard to see that $A= \big( glim ( a_{v,w}(k) ) \big)$ also belongs to the same closed set. Thus, since $ \big( p_k(x,y|v,w) \big)$ is in the closed set $Q^s_x(n,n)$ for all $k,$ we have that $\big( r(x,y|v,w) \big)  \in Q^s_x(n,n).$ Also,  since $p_k$ is a winning $x$-strategy for a graph $x$-homomorphism of $G$, for all $k$, we have that for all $k$,  $\big( p_k(x,y|v,w) \big)$ is zero in certain entries. Since the $glim$ of the 0 sequence is again 0, we will have that $\big( r(x,y|v,w) \big)$ is also 0 in these entries. Hence, $r$ is a winning $x$-strategy for a graph $x$-homomorphism. 
\end{proof} 
\begin{remark} In the case that $p$ is a winning $q$-strategy implementing a graph $q$-homomorphism, all we can say about $r$ is that it is a winning $qa$-strategy implementing a graph $qa$-homomorphism, since we do not know if the set $Q^s_q(n,n)$ is closed.
\end{remark}

There is a natural partial order on idempotent CP maps on $M_n.$  Given two idempotent maps $\phi, \psi: M_n \to M_n$ we set $\psi \le \phi$ if and only if $\psi \circ \phi = \phi \circ \psi = \psi.$

\begin{thm} Let $x \in \{ l, qa, qc, vect \}$, then there exists $ r(x,y|v,w) \in Q^s_x(n,n)$ implementing a quantum $x$-homomorphism, such that
 $\phi_r:M_n \to M_n$  is idempotent and is minimal in the partial order on idempotent maps of the form $\phi_p$ implemented by a quantum $x$-homomorphism of $G$.
\end{thm}

\begin{proof} Quantum $x$-homomorphisms always exist, since the identity map on $G$ belongs to the $l$-homomorphisms, which is the smallest set. By the last theorem we see that beginning with any correlation $p$ implementing a quantum $x$-homomorphism, there exists a correlation $r$ implementing a quantum $x$-homomorphism with $\phi_r$ idempotent. 

It remains to show the minimality claim. We will invoke Zorn's lemma and show that every totally ordered set of such correlations has a lower bound.  Let $\big\{ p_t(x,y|v,w) : t \in T \big\} \subset Q^s_x(n,n)$ with $T$ a totally ordered set, where all $p_t(x,y|v,w)$ implement a quantum $x$-homomorphisms, with  $\phi_{p_t}$ idempotent, and  $\phi_{p_{t}} \le \phi_{p_s},$ whenever $s \le t.$

These define a net in the compact set $Q^s_x(n,n)$ and so we may choose a convergent subnet.
Now it is easily checked that if we define $p(x,y|v,w) $ to be the limit point of this subnet, then it implements a quantum $x$-homomorphism, 
$\phi_p$ is idempotent, and $\phi_p \le \phi_{p_t}$ for all $t \in T.$
\end{proof}

\begin{remark} It is important to note that we are not claiming that $\phi_r$ can be chosen minimal among all idempotent CP maps, just minimal among all such maps that implement a quantum $x$-homomorphism of $G$.
\end{remark}

\begin{defn} Let $x \in \{ l, qa, qc, vect \}$, then a {\bf quantum $x$-core for $G$} is any $ r(x,y|v,w)  \in Q^s_x(n,n)$ that implements a quantum $x$-homomorphism such that $\phi_r$ is idempotent and minimal among all $\phi_p$ implemented by a quantum  $x$-homomorphism of $G.$
\end{defn}


\section*{Appendix: Background Material}

Let $I$ and $O$ be two finite sets called the {\it input} set and {\it output} set, respectively. 

\begin{defn}
A set of real numbers $p(x,y|v,w), \, v,w \in I, \, x,y \in O$ is called a {\bf local} or {\bf classical correlation} if there is a probability space $(\Omega, \mu)$ and random variables, 
\[f_v,g_w: \Omega \to O \text{ for each } v,w\in I\]
such that
\[p(x,y|v,w) = \mu(\{ \omega\ |\ f_v(\omega)=x, g_w(\omega)=y\})\]
\end{defn}

To motivate this definition, imagine that there are two people, Alice and Bob, when Alice receives input $v$ she uses the random variable $f_v$ and when Bob receives input $w$ he uses the random variable $g_w$.
In this case $p(x,y|v,w)$ represents the probability of getting outcomes $x$ and $y$ respectively, given that they received inputs $v$ and $w$, respectively.

\begin{defn} Given a Hilbert space $\cl H$, a collection $\{ E_x : x \in O \}$ of bounded operators on $\cl H$ is called a {\bf projection valued measure(PVM)} on $\cl H$, provided that each $E_x$ is an orthogonal projection and $\sum_{x \in O} E_x = I_{\cl H}.$  The set is called a {\bf positive operator valued measure(POVM)} on $\cl H$, provided that each $E_x$ is a positive semidefinite operator on $\cl H$ and $\sum_{x \in O} E_x = I_{\cl H}.$
\end{defn} 

\begin{defn}
A density $p$ is called a {\bf quantum correlation} if it arises as follows: 

Suppose Alice and Bob have finite dimensional Hilbert spaces $\cl H_A$, $\cl H_B$ and for each input $v \in I$ Alice has PVMs $\{ F_{v,x} \}_{x \in O}$ on $\cl H_A$ and for each input $w \in I$ Bob has PVMs $\{ G_{w,y} \}_{y \in O}$ on $\cl H_B$  and they share a state $\psi \in {\cl H_A \otimes \cl H_B}$,  then
\[ p(x,y|v,w) = \langle  F_{v,x} \otimes G_{w,y} \psi, \psi \rangle\]
This is the probability of getting outcomes $x,y$ given that they conducted experiments $v,w.$ 
\end{defn}

\begin{defn} \label{qcdef}
A density $p$ is called a {\bf quantum commuting correlation} if there is a single Hilbert space $\cl H$,  such that for each $v \in I$ Alice has PVMs $\{ F_{v,x} \}_{x \in O}$ on $\cl H$ and for each $w \in I$ Bob has PVMs $\{G_{w,y} \}_{y \in O} $ on $\cl H$ satisfying  
$$F_{v,x}G_{w,y} = G_{w,y}F_{v,x}, \, \forall v,w,x,y$$  and
\[p(x,y|v,w) = \langle F_{v,x}G_{w,y} \psi, \psi \rangle \] 
where $\psi \in {\cl H}$ is a shared state.\\
\end{defn}

\begin{remark} \label{two correlations}
Suppose we have projection valued measures $\{P_{v,i}\}_{i=1}^m$ and $\{Q_{w,j}\}_{j=1}^m$ on $\cl{H}$ as in \ref{qcdef}. Set $\cl{X}_{v,i}=P_{v,i}k$, $\cl{Y}_{w,j}=Q_{w,j}k$. Then \\
(1) $\cl{X}_{v,i}\perp \cl{X}_{v,j}$ for every $i\neq j$.\\
(2) $\cl{Y}_{w,i}\perp \cl{Y}_{w,j}$ for every $i\neq j$.\\
(3) $\sum_i \cl{X}_{v,i}=\sum_j\cl{Y}_{w,j}$ for every $v, w$ and $\|\sum_i \cl{X}_{v,i}\|=1.$\\
(4) $\inner{\cl{X}_{v,i}}{\cl{Y}_{w,j}}\geq 0$ since $\inner{\cl{X}_{v,i}}{\cl{Y}_{w,j}}=\inner{P^2_{v,i}}{Q^2_{w,j}}=\inner{Q_{w,j}P_{v,i}k}{Q_{w,j}P_{v,i}k}= \|Q_{w,j}P_{v,i}k\|^2\geq0$ where the second equality results from the fact that $Q_{w,j}$ and $P_{v,i}$ are commuting projections.
\end{remark}

\begin{defn}
A density $p$ is called a {\bf vectorial correlation} if $p(i,j|v,w)=\inner{\cl{X}_{v,i}}{\cl{Y}_{w,j}}$ for sets of vectors $\{\cl{X}_{v,i}\},\{\cl{Y}_{w,j}\}$ satisfying (1) through (4) in \ref{two correlations}.
\end{defn}

Letting $n:=|I|$ and $m:=|O|$, we let: 
\begin{itemize}
\item $Q_{loc}(n,m)$ denote the set of all densities that are local correlations.
\item $Q_q(n,m)$ denote the set of all densities that are quantum correlations.
\item Set $Q_{qa}(n,m):= \overline{Q_q(n,m)}$, the closure of $Q_q(n,m)$.
\item $Q_{qc}(n,m)$ denote the set of all densities that are quantum commuting correlations.
\item $Q_{vect}(n,m)$ denote the set of all densities that are vectorial correlations.
\item For $x\in \{ loc, q, qa, qc\}$, we let $Q^s_{x}(n,m)$ denote the set of synchronous correlations in $Q_{x}(n,m)$.
\end{itemize}

\begin{remark}
Results in \cite{PT} and \cite{PSSTW} show that the possibly larger sets that one obtains by using the larger collection of all POVMs in the definitions of $Q_q$, $Q_{qa}$ and $Q_{qc}$ in place of PVMs, yield the same sets. These equalities essentially follow from Stinespring's theorem.  Also, while earlier versions of \cite{PSSTW} use the notation $Q_t(n,m)$, which we have adopted here, this notation was changed to $C_t(n,m)$ in later versions. 
\end{remark}

\begin{remark} In addition to $Q_{vect}(n,m)$ being a natural relaxation of the other sets, determining membership in this set reduces to standard problems in linear algebra.
Another important reason for studying $Q_{vect}(n,m)$ is Tsirelson's 1980 \cite{tsirelson} attempted proof that $Q_q(n,m)= Q_{qc}(n,m)$.  He attempted to show that $Q_q(n,m) = Q_{vect}(n,m)$, from which the other equality would follow, by starting with vectors satisfying (1) through (4) and attempting to build projections $\{P_{v,i}\},\{Q_{w,j}\}$ on finite dimensional Hilbert space, and a vector $k$ such that $\cl{X}_{v,i}=P_{v,i}k$ and $\cl{Y}_{w,j}=Q_{w,j}k$ commuted. In \cite{CMRSSW} a graph on 15 vertices is constructed for which
$\chi_q(G)=8 \ne \chi_{vect}(G)=7$, giving a definitive proof that $Q_{q}(15,7) \ne Q_{vect}(15,7)$, hence showing that for some such set of vectors, one cannot construct corresponding projections. Later, for this same graph \cite{PSSTW} proved that $\chi_{qc}(G) =8 \ne \chi_{vect}(G)$ showing that $Q_{qc}(15,7) \ne Q_{vect}(15,7)$.\\
\end{remark}

Here are some further facts and open problems about these sets that show their importance. 
\begin{itemize}
\item $Q_{loc}(n,m) \subseteq Q_q(n,m) \subseteq Q_{qa}(n,m)  \subseteq Q_{qc}(n,m) \subseteq Q_{vect}(n,m)$. 
\item $Q_{loc}(n,m)$, $Q_{qa}(n,m)$, $Q_{qc}(n,m)$, and $Q_{vect}(n,m)$ are closed.
\item Bounded entanglement conjecture:  $Q_q(n,m)=Q_{qa}(n,m)$  $\forall n,m$, i.e., is $Q_q(n,m)$ closed. 
\item Tsirelson conjecture \cite{tsirelson}:  $Q_{q}(n,m) = Q_{qc}(n,m)\ \forall n,m$.  
\item Ozawa \cite{ozawa} proved that Connes' embedding conjecture \cite{connes} is true if and only if $Q_{qa}(n,m) = Q_{qc}(n,m), \, \forall n,m$. 
\item Paulsen and Dykema \cite{DP} proved that Connes' embedding conjecture is true if and only if $\overline{Q^s_{q}(n,m)} = Q_{qc}^s(n,m), \, \forall n,m$.
\item The synchronous approximation conjecture:  $\overline{Q^s_q(n,m)} = Q^s_{qa}(n,m)$ $\forall n,m$.
\item If Tsirelson's conjecture is true, then the Connes' embedding conjecture and the bounded entanglement conjecture are true.
\item If  Connes' embedding conjecture is true, then the synchronous approximation conjecture is true.
\end{itemize}

\section*{Acknowledgements}

The authors wish to thank S. Severini and D. Stahlke for several
valuable comments that led to improvements in the paper. This research was supported in part by NSF grant DMS-1101231.

\end{document}